\documentclass[11pt]{amsart}
\usepackage[usenames]{color}
\usepackage{fullpage,amscd,amssymb,latexsym,url,hyperref}
\usepackage[all,2cell,ps]{xy}
\usepackage[T1]{fontenc}

\theoremstyle{plain}
\newtheorem{thm}{Theorem}[section]
\newtheorem{theorem}[thm]{Theorem}

\newtheorem{proposition}[thm]{Proposition}

\theoremstyle{definition}
\newtheorem{de}[thm]{Definition}

\newtheorem{remark}[thm]{Remark}
\newtheorem{example}[thm]{Example}
\newtheorem{algorithm}[thm]{Algorithm}

\newcommand{\Z}{\mathbb{Z}}

\newcommand{\bigo}{\mathcal{O}}

\newcommand{\aff}[1]{\mathrm{Aff}(#1)}
\newcommand{\aut}[1]{\mathrm{Aut}(#1)}
\newcommand{\dis}[1]{\mathrm{Dis}(#1)}

\newcommand{\lmlt}[1]{\mathrm{LMlt}(#1)}

\newcommand{\im}[1]{\mathrm{Im}(#1)}

\newcommand{\ld}{\backslash}

\numberwithin{equation}{section}

\begin{document}

\title{Homomorphic images of affine quandles}

\author{P\v remysl Jedli\v cka}
\author{David Stanovsk\'y}

\address{(P.J.) Department of Mathematics, Faculty of Engineering, Czech University of Life Sciences, Kam\'yck\'a 129, 16521 Praha 6, Czech Republic}
\address{(D.S.) Department of Algebra, Faculty of Mathematics and Physics, Charles University, Sokolovsk\'a 83, 18675 Praha 8, Czech Republic}

\email{(P.J.) jedlickap@tf.czu.cz}
\email{(D.S.) stanovsk@karlin.mff.cuni.cz}

\thanks{The second author was partly supported by the GA\v CR grant 13-01832S}


\date{\today}

\begin{abstract}
We are interested in abstract conditions that characterize homomorphic images of affine quandles. Our main result is a two-fold characterization of this class: one by a property of the displacement group, the other one by a property of the corresponding affine mesh. As a consequence, we obtain efficient algorithms for recognizing homomorphic images of affine quandles, including an efficient explicit construction of the covering affine quandle.
\end{abstract}

\maketitle

\section{Introduction}

Quandles are self-distributive structures that appear naturally in the context of knots, braids and in many other situations \cite{EN}. 
Affine quandles (also called Alexander quandles) play a prominent role in quandle theory. From the algebraic perspective, they constitute an important building block \cite{AG,BS,JPSZ}, and it has been observed that many small connected quandles are affine \cite{ESG,G}. In knot theory, there is a close connection between the Alexander invariant and coloring by affine quandles \cite{Bae,Joy}.

In the present paper, we focus on quandles that are \emph{homomorphic images} (or quotients) \emph{of} affine quandles. The paper is, in a way, a blueprint of our recent writeup \cite{JPSZ-qa} on quandles that \emph{embed into} affine quandles. The similarities and differences are outlined below, and explained in detail in Section \ref{sec:main}.

Affine quandles are medial, and so are their subquandles and homomorphic images. Therefore, one can use the representation developed in \cite{JPSZ}, where medial quandles are described using certain heterogeneous affine structures, called affine meshes. For both classes, we have a two-fold characterization (Theorems \ref{thm:embed_affine} and Theorem \ref{thm:homim_affine}): one by a property of the displacement group, the other one by a property of the corresponding affine mesh. 

In both cases, the characterizing conditions are both algorithmically efficient and fairly easy to check for concrete small examples, given either by a mesh, or by a multiplication table. In both proofs, the hard step is, to find an affine quandle such that a given quandle embeds into, resp. projects onto, it. There is, however, one significant difference: for embeddings, our proof is not constructive and we do not know an efficient way to construct such an affine quandle; for quotients, our proof is constructive and turns into a polynomial-time algorithm that constructs the affine quandle.

The two characterization theorems, and several examples illustrating the similarities and differences, are formulated in Section~\ref{sec:main}.
The final Section \ref{sec:proof} contains the proof of the main theorem and an explicit statement of the algorithms based on the proof. 
%

\section{Preliminaries} \label{sec:terminology}

\subsection{Quandles}

For a general introduction to quandle theory we refer to \cite{AG,EN}. The proofs of all statements in this subsection can be found in the introductory part of \cite{HSV} (and also elsewhere, in various notation systems).

We will write mappings acting on the left, hence conjugation in groups will be denoted by $x^y=yxy^{-1}$, and consequently, the commutator will be defined as $[x,y]=y^xy^{-1}=xyx^{-1}y^{-1}$.

A \emph{quandle} is an algebraic structure $(Q,*)$ which is \emph{idempotent} (it satisfies the identity $x*x=x$) and in which all left translations, $L_x(y)=x*y$, are automorphisms.
The unique $y$ such that $x*y=z$ will be denoted $y=x\ld z$. There are two important permutation groups associated to every quandle:
the (left) \emph{multiplication group}, generated by all left translations,
\[\lmlt Q=\langle L_a:\ a\in Q\rangle\leq\aut Q,\]
and its subgroup, the \emph{displacement group}, defined by
\[\dis Q=\langle L_aL_b^{-1}:\ a,b\in Q\rangle\leq\lmlt Q.\]
It is easy to see that $\dis Q=\langle L_aL_e^{-1}:\ a\in Q\rangle$ for any fixed $e\in Q$.
Both groups have the same orbits of the natural action on $Q$, to be called \emph{orbits} of the quandle $Q$, and denoted 
\[Qe=\{\alpha(e):\ \alpha\in\lmlt Q\}=\{\alpha(e):\ \alpha\in\dis Q\}.\]
Orbits are subquandles of $Q$. They form a block system, to be called the \emph{orbit decomposition} of $Q$.

Observe that $L_{\alpha(x)}=(L_x)^\alpha$ for every automorphism $\alpha$. Consequently, both $\lmlt Q$ and $\dis Q$ are normal in $\aut Q$.

Let $\lambda=\{(a,b):L_a=L_b\}$ denote the \emph{Cayley kernel} of a quandle $Q$. This is always a congruence on $Q$, since it is the kernel of the quandle homomorphism $a\mapsto L_a$ from $Q$ into the symmetric group $S_Q$ under the conjugation operation.

A quandle is called \emph{medial} if it satisfies the identity $(x*y)*(u*v)=(x*u)*(y*v)$. 
A quandle is medial if and only if its displacement group is abelian. 

Let $(A,+)$ be an abelian group, $f$ its automorphism, and define an operation on the set $A$ by
\[a*b=(1-f)(a)+f(b).\]
Then $(A,*)$ is a medial quandle, to be denoted $\aff{A,f}$, and called \emph{affine} over the group $(A,+)$.
Here $1$ refers to the identity mapping, hence $g=1-f$ is the mapping $g(x)=x-f(x)$.

Let $Q=\aff{A,f}$. Then $\dis Q\simeq\im{1-f}$, where $a\in\im{1-f}$ corresponds to the mapping $x\mapsto a+x$ (indeed, $L_aL_b^{-1}(x)=(1-f)(a-b)+x$).
Hence, the orbits of $Q$ are the cosets of $\im{1-f}$. 


\subsection{Affine meshes}

In \cite{JPSZ}, we developed a representation of medial quandles by affine meshes. 
Here we recall the essential constructions and results.

\begin{de}
An \emph{affine mesh} over a non-empty set $I$ is a triple
$$\mathcal A=((A_i)_{i\in I};\,(\varphi_{i,j})_{i,j\in I};\,(c_{i,j})_{i,j\in I})$$ where $A_i$ are abelian groups, $\varphi_{i,j}:A_i\to A_j$ homomorphisms, and $c_{i,j}\in A_j$ constants, satisfying the following conditions, for every $i,j,j',k\in I$:
\begin{enumerate}
    \item[(M1)] $1-\varphi_{i,i}$ is an automorphism of $A_i$;
    \item[(M2)] $c_{i,i}=0$;
    \item[(M3)] $\varphi_{j,k}\varphi_{i,j}=\varphi_{j',k}\varphi_{i,j'}$, i.e., the following diagram commutes:
$$\begin{CD}
A_i @>\varphi_{i,j}>> A_j\\ @VV\varphi_{i,j'}V @VV\varphi_{j,k}V\\
A_{j'} @>\varphi_{j',k}>> A_k
\end{CD}$$
    \item[(M4)] $\varphi_{j,k}(c_{i,j})=\varphi_{k,k}(c_{i,k}-c_{j,k})$.
\end{enumerate}
The mesh is called \emph{indecomposable} if,
for every $j\in I$, the group $A_j$ is generated by all the elements $c_{i,j}$, $\varphi_{i,j}(a)$ with $i\in I$ and $a\in A_i$.
\end{de}

If the index set is clear from the context, we shall write briefly $\mathcal A=(A_i;\varphi_{i,j};c_{i,j})$.

\begin{de}
The \emph{sum of an affine mesh} $(A_i;\varphi_{i,j};c_{i,j})$ is an algebraic structure $(A,*)$ defined on the disjoint union of the sets $A_i$ by
\[ a\ast b=c_{i,j}+\varphi_{i,j}(a)+(1-\varphi_{j,j})(b). \]
for every $a\in A_i$ and $b\in A_j$.
\end{de}

The sum of any affine mesh is a medial quandle, with $a\ld b=(1-\varphi_{j,j})^{-1}(b-\varphi_{i,j}(a)-c_{i,j})$.
The fibers $A_i$ form subquandles which are affine, namely, $\aff{A_i,1-\varphi_{i,i}}$.
If the mesh is indecomposable, the fibers coincide with the orbits.

\begin{theorem}\cite{JPSZ}\label{thm:decomposition}
A binary algebraic structure is a medial quandle if and only if it is the sum of an indecomposable affine mesh.
\end{theorem}

%

\section{The two characterization theorems}\label{sec:main}

The key to recognition of quandles that embed into an affine quandle is the \emph{lack of fixed points} (i.e., semiregularity) in the displacement group, which translates into a certain form of homogenity of the underlying meshes. The following is a reformulation of the main result of \cite{JPSZ-qa} (the special type of meshes in condition (3) was called a \emph{semiregular extension}).

\begin{theorem}\cite{JPSZ-qa}\label{thm:embed_affine}    
The following statements are equivalent for a quandle $Q$:
\begin{enumerate}
	\item $Q$ embeds into an affine quandle;
	\item $\dis Q$ is abelian and semiregular;
	\item there is an abelian group $A$, an automorphism $\psi$ of $A$ and elements $d_i\in A$ such that 
	$Q$ is isomorphic to the sum of an affine mesh $((A_i);(\varphi_{i,j});(c_{i,j}))$ where $A_i=A$ for every $i$, and $\varphi_{i,j}=1-\psi$ and $c_{i,j}=d_i-d_j$ for every $i,j\in I$.
\end{enumerate}
\end{theorem}

Given a multiplication table of a quandle $Q$, it is easy to verify condition (2). Given an affine mesh, it is easy to verify condition (3). However, in either case, we do not know how to find efficiently an abelian group $A$ and its automorphism $f$ such that $Q$ embeds into $\aff{A,f}$.

The key to recognition of quandles that are quotients of affine quandles is the \emph{size} of the displacement group, which translates into a certain form of linearity of the underlying meshes. 

\begin{de}
We will say that a quandle $Q$ has a \emph{tiny displacement group} if, for some $e\in Q$, \[ \dis Q=\{L_xL_e^{-1}:x\in Q\}. \]
\end{de}

Note that if $\dis Q$ is tiny, then $\dis Q=\{L_xL_f^{-1}:x\in Q\}$ for \emph{every} $f\in Q$: to express $L_xL_e^{-1}$ as $L_yL_f^{-1}$ for some $y$, consider $L_xL_e^{-1}L_fL_e^{-1}\in\dis Q$ and take $y$ such that it equals $L_yL_e^{-1}$. 

The following is the main result of the present paper.

\begin{theorem}\label{thm:homim_affine}
The following statements are equivalent for a quandle $Q$:
\begin{enumerate}
	\item $Q$ is a homomorphic image of an affine quandle;
	\item $\dis Q$ is abelian and tiny;
	\item $Q$ is the sum of an affine mesh $((A_i);(\varphi_{i,j});(c_{i,j}))$ such that 
	$$\{ (\varphi_{i,j}(a)+c_{i,j})_{j\in I}:\ i\in I,\, a\in A_i\}\subseteq\prod_{j\in I} A_j$$
	is a coset of a subgroup of $\prod A_j$.
\end{enumerate}
\end{theorem}

Given a multiplication table of a quandle $Q$, it is easy to verify condition (2). Given an affine mesh, it is easy to verify condition (3) (a subset $X$ of a group $G$ is a coset if and only if, for any $h\in X$, $-h+X$ is a subgroup of $G$). As we shall see in the next section, the proof of Theorem \ref{thm:homim_affine} provides an efficient algorithm that constructs an abelian group $A$ and its automorphism $f$ such that $\aff{A,f}$ maps homomorphically onto $Q$.

Condition (3) is particularly easy to use for 2-reductive medial quandles, i.e., quandles given by meshes with $\varphi_{i,j}=0$ for all $i,j$ \cite[Section 6]{JPSZ}. Indeed, the sum of a mesh $(A_i,0,(c_{i,j}))$ is a homomorphic image of an affine quandle if and only if the rows of the matrix $(c_{i,j})$ form a coset in the group $\prod A_j$.

\begin{example}
All of the following properties are easy to verify using conditions (3) in Theorems \ref{thm:embed_affine} and \ref{thm:homim_affine}.
\begin{itemize}
	\item The sum of the affine mesh \[ \left((\Z_2,\Z_2,\Z_2);\ \left( \begin{smallmatrix} 0 & 0 & 0\\ 0 & 0 & 0 \\ 0 & 0 & 0\end{smallmatrix}\right);\ \left(\begin{smallmatrix} 0&0&1\\0&0&1\\1&1&0\end{smallmatrix}\right)\right) \]
	both embeds into an affine quandle (for instance into $\aff{\Z_8,5}$ as $\{0,4,2,6,1,5\}$), and it is a homomorphic image of an affine quandle (for example of
	$\aff{\Z_8,5}$ over the congruence $\{0\},\{2\},\{4\},\{6\},\{1,3\},\{5,7\}$). However, it is not affine, since it has three orbits, unlike any of the two 6-element affine quandles.
	\item The sum of the affine mesh \[ \left((\Z_3,\Z_3);\ \left( \begin{smallmatrix} 0 & 0 \\ 0 & 0\end{smallmatrix}\right);\ \left(\begin{smallmatrix} 0&1\\1&0\end{smallmatrix}\right)\right) \]
	embeds into an affine quandle (for instance into $\aff{\Z_9,4}$ as $\{0,3,6,1,4,7\}$), but it is not a homomorphic image of an affine quandle.
	\item The sum of the affine mesh \[ \left((\Z_2,\Z_1);\ \left( \begin{smallmatrix} 0 & 0 \\ 0 & 0\end{smallmatrix}\right);\ \left(\begin{smallmatrix} 0&0\\1&0\end{smallmatrix}\right)\right) \]
	is a homomorphic image of an affine quandle (for example of $\aff{\Z_4,-1}$ over the congruence $\{0\},\{2\},\{1,3\}$), but does not embed into an affine quandle since
	it has orbits of different sizes.
\end{itemize}
\end{example}

\begin{example}
 We calculate all 2-reductive quandles with two orbits that are homomorphic images of affine quandles. Such quandles are sums of indecomposable affine meshes of type
 \[\left( (A,B);\ \left( \begin{smallmatrix} 0 & 0 \\ 0 & 0\end{smallmatrix}\right);\ \left( \begin{smallmatrix} 0 & b \\ a & 0\end{smallmatrix}\right) \right)\]
where $A=\langle a\rangle$ and $B=\langle b\rangle$. Condition (3) of Theorem \ref{thm:homim_affine} states that the rows of the matrix~$(c_{i,j})$ form a coset in $A\times B$. The coset is necessarily $(a,0)+\langle(-a,b)\rangle$, hence $2a=2b=0$, $A,B\in\{\Z_1,\Z_2\}$ and there are only three options, up to isomorphism:
 \[
 \left( (\Z_1,\Z_1);\ \left( \begin{smallmatrix} 0 & 0 \\ 0 & 0\end{smallmatrix}\right)
 ;\ \left( \begin{smallmatrix} 0 & 0 \\ 0 & 0\end{smallmatrix}\right) \right), \qquad
 \left( (\Z_2,\Z_1);\ \left( \begin{smallmatrix} 0 & 0 \\ 0 & 0\end{smallmatrix}\right)
 ;\ \left( \begin{smallmatrix} 0 & 0 \\ 1 & 0\end{smallmatrix}\right) \right), \quad
 \left( (\Z_2,\Z_2);\ \left( \begin{smallmatrix} 0 & 0 \\ 0 & 0\end{smallmatrix}\right)
 ;\ \left( \begin{smallmatrix} 0 & 1 \\ 1 & 0\end{smallmatrix}\right) \right).\]
\end{example}

\section{The proof and the algorithm}\label{sec:proof}



Informally, a \emph{multiset} is a generalization of the notion of a set where elements can repeat. Tuples can be considered as multisets, forgetting the indexing. A \emph{multitransversal} for a block system is a multiset which takes the same amount of elements from each block (i.e., a multiset $T$ such that $|T\cap B_1|=|T\cap B_2|$, for every pair of blocks $B_1,B_2$).
The {\em multiplicity} of a multitransversal is
the cardinality of each such $T\cap B_i$.

\begin{proof}[Proof of Theorem \ref{thm:homim_affine}]
$(1)\Rightarrow(2)$ Affine quandles satisfy (2) and both properties carry over to homomorphic images.

$(2)\Rightarrow(1)$ 
We shall construct a group~$A$ and an automorphism~$f$ such that $\aff{A,f}$ maps homomorphically onto $Q$.
Let $T$ be a multitransversal to the Cayley kernel $\lambda=\{(a,b):L_a=L_b\}$ which contains at least one element from each orbit of $Q$ (take a transversal, add a representative of every orbit, and increase multiplicity of selected elements to obtain a multitransversal). Let $\kappa$ be the multiplicity of $T$.

We will treat the elements of~$T$ as formally different and construct an abelian group operation on~$T$. Choose an element~$e\in T$ which will play the role of zero.
Consider an arbitrary abelian group $K=(K,+,-,0)$ of order $\kappa$ and an arbitrary mapping $\nu:T\to K$ which is bijective on every block of $T^2\cap\lambda$ and satisfies $\nu(e)=0$.
Define an operation $\oplus$ on~$T$ by
\[a\oplus b=c\ \Leftrightarrow\ L_aL_e^{-1}L_b=L_c\text{ and }\nu(a)+\nu(b)=\nu(c).\]
The operation is well defined: $\alpha=L_aL_e^{-1}L_bL_e^{-1}$ is in $\dis Q$ which is tiny, hence there exists $c\in Q$ such that $\alpha=L_cL_e^{-1}$ and among the $\kappa$ candidates for $c$ in~$T$, there is a unique one with $\nu(c)=\nu(a)+\nu(b)$. 

Clearly, $e$ is  a unit element for $\oplus$. An inverse to $a$ is an element $b$ such that $L_aL_e^{-1}L_b=L_e$ and $\nu(a)+\nu(b)=0$, that is, $b$ such that $L_bL_e^{-1}=(L_aL_e^{-1})^{-1}$ and $\nu(b)=-\nu(a)$; such $b$ exists because $\dis Q$ is tiny. The operation $\oplus$ is commutative because both $\dis Q$ and $K$ are abelian. It is associative, because $d=a\oplus(b\oplus c)$ if and only if $L_d=L_aL_e^{-1}L_{b\oplus c}=L_aL_e^{-1}L_bL_e^{-1}L_c$ and $\nu(d)=\nu(a)+\nu(b\oplus c)=\nu(a)+\nu(b)+\nu(c)$; the two rightmost expressions do not depend on the bracketing.

Now, let \[A=\dis{Q}\times(T,\oplus)\] and consider the mapping \[f:A\to A,\quad (\alpha,a)\mapsto(L_e\alpha L_a^{-1},a).\]
Then $f$ is an endomorphism of $A$, because
\begin{multline*}
f((\alpha,a))+f((\beta,b))=(L_e\alpha L_a^{-1},a)+(L_e\beta L_b^{-1},b)
=(L_e\alpha\beta L_a^{-1}L_eL_b^{-1},a\oplus b)=\\
=(L_e\alpha\beta L_e^{-1}(L_bL_e^{-1}L_aL_e^{-1})^{-1},a\oplus b)=
(L_e\alpha\beta L_e^{-1}(L_{b\oplus a}L_e^{-1})^{-1},a\oplus b)=f(\alpha\beta,a\oplus b).
\end{multline*}
The kernel of~$f$ is trivial: $\mathrm{Ker}(f)=\{(\alpha,a):L_e\alpha L_a^{-1}=1$ and $a=e\}=\{(1,e)\}$. To show that $f$ is onto, notice that $(\alpha,a)=f((L_e^{-1}\alpha L_a,a))$ where 
$L_e^{-1}\alpha L_a=L_e^{-1}(\alpha L_aL_e^{-1})L_e\in\dis Q$ since it is normal in $\lmlt Q$. Hence $f$ is an automorphism of $A$.

Finally, consider \[\psi:\aff{A,f}\to Q,\quad (\alpha,a)\mapsto \alpha(a).\] 
We have
\begin{align*}
\psi((\alpha,a)*(\beta,b))&=\psi((\alpha,a)f((\alpha,a))^{-1}f((\beta,b)))\\ &=\psi((\alpha L_a\alpha^{-1}L_e^{-1}L_e\beta L_b^{-1},b))=
\alpha L_a\alpha^{-1}\beta L_b^{-1}(b)=L_{\alpha(a)}\beta(b)=\alpha(a)*\beta(b),
\end{align*}
hence $\psi$ is a homomorphism. It is onto, because each orbit of $Q$ contains at least one element in $T$.

$(2)\Leftrightarrow(3)$ 
Assume that $Q$ is the sum of an affine mesh $((A_i);(\varphi_{i,j});(c_{i,j}))$. The displacement mapping $L_aL_b^{-1}$ with $a\in A_i$, $b\in A_j$ can be expressed as
\[L_aL_b^{-1}(x)=x+\varphi_{i,k}(a)-\varphi_{j,k}(b)+c_{i,k}-c_{j,k} \text{ whenever }x\in A_k. \]
We will prove that $\dis Q$ is tiny if and only if $\{ (\varphi_{i,k}(a)+c_{i,k})_{k\in I}:\ i\in I,\, a\in A_i\}$ is a coset in $\prod A_k$, that is, if and only if, for some $e\in A_j$, the set 
\[ S_e=-(\varphi_{j,k}(e)+c_{j,k})_{k\in I}+\{ (\varphi_{i,k}(a)+c_{i,k})_{k\in I}:\ i\in I,\, a\in A_i\} \] is a subgroup of $\prod A_k$.

Fix $e\in A_j$. Then the set $\Lambda_e=\{L_aL_e^{-1}:a\in Q\}$	is in 1-1 correspondence with the set $S_e$, where $L_aL_e^{-1}$ corresponds to the tuple $=(-\varphi_{j,k}(e)-c_{j,k}+\varphi_{i,k}(a)+c_{i,k})_{k\in I}$. Moreover, composition of mappings from $\Lambda_e$ corresponds to addition of the corresponding elements of $S_e$. Therefore, $\Lambda_e$ is closed with respect to composition and inversion if and only if $S_e$ is closed with respect to addition and subtraction.

Consequently, $\dis Q$ is tiny if and only if $\dis Q=\Lambda_e$ for some $e$, which is equivalent to $S_e\leq\prod A_k$.
\end{proof}


\def\krok#1#2{#1 & #2 \\}
\def\io#1#2{{$\ $}\\ \begin{tabular}{ll} {\bf In:} & #1 \\ {\bf Out:} & #2 \end{tabular} \\}
\def\afont#1{{\bf #1}}

From an algorithmic point of view, the size of the multitransversal $T$ is important. Indeed, $|T|=|Q/\lambda|\cdot\kappa$, and we can always take $\kappa$ at most the number of orbits. Both values are bounded by the size of $Q$, hence we can always find $T$ such that $|T|\leq |Q|^2$. The following example shows that there is also a quadratic lower bound. 

\begin{example}\label{ex:max_quandle}
Let $n,k$ be natural numbers such that $2^k<n$. Consider the affine mesh $\mathcal A_{n,k}=(Z;0;C)$ where 
$Z=(\Z_2,\dots,\Z_2,\Z_1,\dots,\Z_1)$ contains $k$ copies of $\Z_2$ and $n-k$ copies of $\Z_1$, $0$ is the matrix of zero homomorphisms, and 
\[ C=\left(\begin{array}{c|c}
D & 0 \\ \hline
0 & 0 
\end{array}\right) \]
where $D$ is an arbitrary $2^k\times k$ matrix over $\Z_2$ whose set of row vectors equals $\Z_2^k$. To simplify notation, assume that the zero vector is in the last row.

Clearly, $\mathcal A_{n,k}$ is an indecomposable affine mesh and using Theorem~\ref{thm:homim_affine}, we see that its sum, $Q$, is a homomorphic image of an affine quandle.
The quandle $Q$ has $n$ orbits and $|Q|=n+k$.

Observe that the Cayley kernel of $Q$ consists of all pairs $(a,b)$ where $a$ is in the $i$-th orbit, $b$ is in the $j$-th orbit, and the $i$-th and $j$-th row in $C$ are equal. 
Therefore, $\lambda$ has $2^k$ blocks. Most of them have 1 or 2 elements, and there is one large block $B$ of size $n-2^k+1$, consisting of singleton orbits with indices $2^k,\dots,n$.
Consequently, any multitransversal $T$ intersectiong all orbits must choose all elements from $B$, thus its multiplicity $\kappa$ must be at least $n-2^k+1$ and we get 
\[ |T|\geq|Q/\lambda|\cdot\kappa=2^k\cdot(n-2^k+1) \]
(the lower bound can be achieved). In particular, if $n=2^{k+1}$, we have $|T|\geq\frac n2\cdot(\frac n2+1)\approx|Q|^2/4$.
\end{example}

In the rest of the section, we present two algorithms.
The input is a finite quandle, in any form that allows efficient calculation of left translations (for example, the multiplication table, or the corresponding affine mesh). 
In the first one, we are asked to decide whether it is a homomorphic image of an affine quandle. In the second one, we are asked to find such an affine quandle.

\begin{algorithm}\label{alg:homim-affine}
\noindent
\io{a finite quandle $Q$}{is $Q$ a homomorphic image of an affine quandle?}
\begin{tabular}{ll}
\krok{1.}{pick $e\in Q$ }
\krok{2.}{$D:=\{L_xL_e^{-1}:x\in Q\}$}
\krok{3.}{\afont{for each} $\alpha,\beta\in D$ \afont{do}}
\krok{4.}{\quad \afont{if} $\alpha\beta\neq\beta\alpha$ or $\alpha\beta\not\in D$ \afont{then return} false}
\krok{5.}{\afont{return} true}
\end{tabular}
\end{algorithm}

On line 2, we define the group $\dis Q$. On lines 3 and 4, we verify condition (2) of Theorem \ref{thm:homim_affine}: if we find a non-commuting pair, or if we find a pair whose composition is not inside $D$, the algorithm reports a failure.

\begin{proposition}\label{p:alg-homim-affine}
Algorithm \ref{alg:homim-affine} runs in a polynomial time with respect to $n=|Q|$, namely $\bigo(n^4\log n)$ (i.e., in a quadratic time with respect to the input size). 
\end{proposition}

\begin{proof}
All operations performed with permutations on $Q$ (comparison, composition) can be calculated in $\bigo(n\log n)$ time. The set $D$ has at most $n$ elements, hence the loop on lines 3--5 performs at most $n^2$ steps. Checking commutativity takes $\bigo(n\log n)$ time (two compositions, one comparison), checking containment in $D$ takes $\bigo(n^2\log n)$ time (one composition, at most $n$ comparisons).
\end{proof}


\begin{algorithm}\label{alg:homim-affine2}
\noindent
\io{a finite quandle $Q$}{$(A,f)$ such that $Q\simeq\aff{A,f}$, or failure if such $(A,f)$ does not exist}
\begin{tabular}{ll}
\krok{1.}{pick $e\in Q$ }
\krok{2.}{$D:=\{L_xL_e^{-1}:x\in Q\}=\{\alpha_0,\dots,\alpha_{m-1}\}$ where $\alpha_0=1$}
\krok{3.}{\afont{for each} $\alpha,\beta\in D$ \afont{do}}
\krok{4.}{\quad \afont{if} $\alpha\beta\neq\beta\alpha$ or $\alpha\beta\not\in D$ \afont{then stop} with failure}
\krok{5.}{\afont{for each} $0\leq i<m$ \afont{do}}
\krok{6.}{\quad find all $x_{i,0},\dots,x_{i,k_i-1}$ such that $L_{x_{i,j}}L_e^{-1}=\alpha_i$ (for $i=0$, take $x_{0,0}=e$)}
\krok{7.}{$\kappa:=\max\{k_i\}$}
\krok{8.}{\afont{for each} $0\leq i<m$ and $0\leq j < \kappa$ \afont{do}}
\krok{9.}{\quad $T_{i\cdot \kappa +j}:=x_{i,j \bmod k_i}$}
\krok{10.}{\afont{for each} $0\leq i,i'<m$ and $0\leq j,j'<\kappa$ \afont{do}}
\krok{11.}{\quad set $T_{i\cdot \kappa+j}\oplus T_{i'\cdot\kappa+j'}:=T_{i''\cdot\kappa+((j+j')\bmod \kappa)}$ such that $L_{x_{i'',0}}L_e^{-1}=L_{x_{i,0}}L_e^{-1}L_{x_{i',0}}L_e^{-1}$}
\krok{12.}{\afont{return} $(D\times T,f)$ where $f(\alpha,T_{i\cdot\kappa+j})=(L_e\alpha L_{x_{i,j}}^{-1},T_{i\cdot\kappa+j})$}
\end{tabular}
\end{algorithm}

On lines 5 and 6 we calculate the blocks of the Cayley kernel $\lambda$.
To keep things simple and to avoid calculation of the orbits of $Q$, we put every element of~$Q$ in~$T$ (occasionally several times).
We implicitly choose the group~$K$ to be the cyclic group $\Z_\kappa$. The construction of the group $(T,\oplus)$ with $T=\{T_k:k=0,\dots,m\kappa-1\}$ on lines 10 and 11 follows the proof of Theorem \ref{thm:homim_affine}.

\begin{proposition}\label{p:alg-homim-affine2}
Algorithm \ref{alg:homim-affine2} runs in a polynomial time with respect to $n=|Q|$, namely $\bigo(n^6\log n)$ (i.e., in a cubic time with respect to the input size). 
\end{proposition}

\begin{proof}
The most costly part is the cycle on lines 10 and 11 requiring $m^2\kappa^2$ steps, each of complexity $m\cdot n\log n$
(search for a permutation in a list of length $m$). In the worst case, we have to assume both $m,\kappa=\Theta(n)$.
\end{proof}

Our choice of $T$ is simple, but often not optimal in terms of size. Here we outline a better approach. 
First, construct a subset $S\subseteq Q$ which is a transversal of the orbit decomposition, and minimizes the maximal number of elements taken from any single block of $\lambda$; shortly,  \[\min_S \max_{B\in Q/\lambda} |B\cap S|. \]
Given $S$, we obtain $T$ by adding a proper amount of arbitrarily chosen elements from each block of $\lambda$. Finding the optimal set $S$ can be formulated as an instance of integer linear programming. Let $A_1,\dots,A_n$ be the orbits of $Q$, and $B_1,\dots,B_m$ the blocks of $\lambda$. Let $x_{i,j}$ be the indeterminates that tell how many elements we choose from $A_i\cap B_j$. We set $x_{i,j}=0$ whenever the two sets are disjoint. The constraints are $x_{i,j}\geq 0$ and $\sum_{j=1}^m x_{i,j}=1$ for every $i=1,\dots,n$ (thus we choose exactly one element from each $A_i$). We minimize $c$ such that $\sum_{i=1}^n x_{i,j}\leq c$ for every $j=1,\dots,m$.
While integer linear programming is a difficult problem in general, there are efficient heuristics for finding good solutions.

For the quandle from Example \ref{ex:max_quandle}, our choice of $T$ is optimal. Therefore, the worst case asymptotic complexity of our algorithm cannot be improved by a better choice of $T$.

\begin{remark}
In \cite{JPSZ-qa} we described an efficient algorithm to recognize quandles isomorphic to affine quandles, but we do not know how to find efficiently the actual affine representation (i.e., the group and its automorphism).
The construction from the proof of Theorem \ref{thm:homim_affine} does not help either. The homomorphism $\psi:\aff{A,f}\to Q$ constructed in the proof is bijective if and only if the multitransversal $T$ has precisely one element from each orbit of $Q$. However, many affine quandles do not admit such $T$: for example, every affine latin quandle $Q$ has only one orbit, but $|T|\geq|Q|$.
\end{remark}



\subsection*{Acknowledgement}
The authors wish to dedicate the paper to memory
of Patrick Dehornoy who was one of the most inspiring
guides through the world of self-distributivity.


\begin{thebibliography}{99}

\bibitem{AG} N. Andruskiewitsch, M. Gra\~na,\emph{\ From racks to pointed Hopf algebras}, Advances in Math. 178/2 (2003), 177--243.

\bibitem{Bae}
Y. Bae, \emph{Coloring link diagrams by Alexander quandles}, J. Knot Theory Ramifications 21/10 (2012), 1250094, 13 pp.

\bibitem{BS}
M. Bonatto, D. Stanovsk\'y, \emph{Commutator theory for racks and quandles,} to appear in J. Math. Soc. Japan, \url{https://arxiv.org/abs/1902.08980}.

\bibitem{EN}
M. Elhamdadi, S. Nelson, \emph{Quandles: an introduction to the algebra of knots.} Student Mathematical Library, 74. American Mathematical Society, Providence, RI, 2015.

\bibitem{ESG} P. Etingof, A. Soloviev, R. Guralnik, \emph{Indecomposable set-theoretical solutions to the Quantum Yang-Baxter Equation on a set with prime number of elements}, Journal of Algebra 242 (2001), 709--719.

\bibitem{G}
M. Gra\~na, \emph{Indecomposable racks of order $p^2$}, Beitr\"age Algebra Geom. 45 (2004), no. 2, 665--676.


\bibitem{Hou}
X. Hou, \emph{Finite modules over $\Z[t,t^{-1}]$}, J. Knot Theory Ramifications 21 (2012), no. 8, 1250079, 28 pp.

\bibitem{HSV}
A. Hulpke, D. Stanovsk\'y, P. Vojt\v echovsk\'y, {\it Connected quandles and transitive groups}, J. Pure Appl. Algebra 220 (2016), 735--758.

\bibitem{JPSZ}
P. Jedli\v cka, A. Pilitowska, D. Stanovsk\'y, A. Zamojska-Dzienio, \emph{The structure of medial quandles}, J. Algebra 443 (2015), 300--334.

\bibitem{JPSZ-qa}
P. Jedli\v cka, A. Pilitowska, D. Stanovsk\'y, A. Zamojska-Dzienio, \emph{Subquandles of affine quandles},
J. Algebra 510,15 (2018), 259--272


\bibitem{Joy}
D. Joyce, \emph{Classifying invariant of knots, the knot quandle}, J. Pure Appl. Algebra, 1982, 23, 37--65.

\end{thebibliography}
\end{document}